\documentclass[a4paper,12pt]{amsart}
\usepackage{latexsym}
\usepackage{amsmath,amssymb}
\usepackage{amsmath}

\usepackage[all]{xy}
\usepackage{enumerate}
\usepackage[usenames]{color}
\usepackage{xcolor}

\usepackage{amsmath, epsfig, cite}

\theoremstyle{plain}
\newtheorem{theorem}{Theorem}[section]

\newtheorem{corollary}[theorem]{Corollary}
\newtheorem{proposition}[theorem]{Proposition}
\theoremstyle{definition}

\newtheorem{example}[theorem]{Example}
\newtheorem{remark}[theorem]{Remark}

\makeatletter
\@namedef{subjclassname@2020}{%
  \textup{2020} Mathematics Subject Classification}
\makeatother

\begin{document}
\sloppy

\title[Hyperreals and Topological Filters]{A Connection between Hyperreals\\ and Topological Filters}

\author[M. Benslimane]{Mohamed Benslimane}
\address{Department of Mathematics, Faculty of Sciences,\\ Abdelmalek Essa\^{a}di University, B.P. 2121 Tetouan, Morocco} \email{med.bens@gmail.com}



\begin{abstract}
Let $U$ be an absolute ultrafilter on the set of non-negative integers $\mathbb{N}$.
For any sequence $x=(x_n)_{n\geq 0}$ of real numbers, let $U(x)$ denote the topological filter consisting of the open sets $W$ of $\mathbb{R}$ with
$\{n \geq 0, x_n \in W\} \in U$. It turns out that for every $x \in \mathbb{R}^{\mathbb{N}}$, the hyperreal $\overline{x}$ associated to $x$ (modulo $U$)
is completely characterized by $U(x)$. This is particularly surprising.
We introduce the space $\widetilde{\mathbb{R}}$ of saturated topological filters of $\mathbb{R}$ and then we prove that the set $^\ast\mathbb{R}$ of hyperreals modulo $U$  can be embedded in $\widetilde{\mathbb{R}}$. It is also shown that $\widetilde{\mathbb{R}}$ is quasi-compact
and that $^\ast\mathbb{R} \setminus \mathbb{R}$ endowed with the induced topology by the space $\widetilde{\mathbb{R}}$ is a separated topological space.
\end{abstract}

\date{May 15, 2024}

\maketitle

\section{Introduction}

Following the research papers which have been devoted to the study of \textquotedblleft hyperreals\textquotedblright
or \textquotedblleft topological filters\textquotedblright, it seems that these two notion are independent.
The present paper aims to establish a bridge between these two concepts. An additional point which motivated our study is that we expect that this connection
may lead to construct a mathematical tool that could unify the corpuscular and probabilistic theories of particle physics.

We first recall that a nonempty collection $\theta$ of open subsets
of a topological space $(E; \tau)$ is said to be a {\it $\tau$-filter} of $E$ if $\theta$ satisfies the following three conditions:

(i) $\emptyset \not\in \theta$.

(ii) If $A, B \in \theta$, then  $A \cap B \in \theta.$

(iii) If $A \in \theta$ and $A \subseteq B \in \tau$, then $B \in \theta.$

In this paper we introduce the notion of saturated filters. If a $\tau$-filter $\theta$ satisfies also the condition

\begin{center}
(iv) If $A, B \in \tau$ with $A \cup B \in \theta$, then  $A \in \theta$ or $B \in \theta,$
\end{center}
then $\theta$ is said to be a {\it saturated filter} ({\it $s$-filter}, for short). The set of all $s$-filters of $E$ will be denoted by $\widetilde{E}$.

\begin{example} \label{s-filter} \rm Let $(E; \tau)$ be a topological space and let $x \in E$. It is easily seen that the set $\omega(x)$
of open neighbourhoods of $x$ is an $s$-filter called a {\it principal $s$-filter}.
\end{example}
The set of all principal $s$-filters of $E$ is denoted by $[E]$. The paper is organized as follows.
In Section $1$, several interesting facts about $s$-filters are established.
Among others, we show that the space $\widetilde{\mathbb{R}}$ is quasi-compact (Proposition \ref{quasi-compact-2}).
In Section $2$, we discover an important connection between hyperreals and topological filters (Theorem \ref{converse}).
We refer the interested reader to \cite{Goldblatt} for definitions and basic facts about hyperreals.
We conclude the paper by proving that $^\ast\mathbb{R} \setminus \mathbb{R}$ endowed with the induced topology by the space $\widetilde{\mathbb{R}}$
is a separated topological space (Theorem \ref{separated}). \\
\noindent {\bf Notations} \\
In this paper we will use the following notations.
The letters $\mathbb{N}$ and $\mathbb{R}$ are used for the set of non-negative integers and the field of real numbers, respectively.
For given subsets $A$ and $B$ of a set $X$, $A \setminus B=\{a \in A \mid a \not\in B\}$, and we write $A^c=X \setminus A$.
We use $|X|$ to denote the cardinality of a set $X$. Given an absolute ultrafilter $U$ on $\mathbb{N}$, we let $^\ast\mathbb{R}$
denote the set of hyperreals modulo $U$.

\section{Some properties of saturated filters}

In this section we establish various interesting facts about $s$-filters which will be used in the last section to achieve our objective.
We begin with the following result which shows that every topological space possesses non-principal $s$-filters.

\begin{proposition} \label{maximal-filter} Every maximal $\tau$-filter of a topological space $E$ is an $s$-filter.
\end{proposition}

\begin{proof} Let $E$ be a topological space and let $\rho$ be a maximal $\tau$-filter of $E$. Consider the set
$\mu = \{X \in \tau \mid X \cap A \neq \emptyset$ for every $A \in \rho\}$. We claim that $\mu \subseteq \rho$. To see this, let $X \in \mu$.
It is easy to check that the set
\begin{center}
$\sigma = \{B \in \tau \mid X \cap A \subseteq B$ for some $A \in \rho\}$
\end{center}
is a $\tau$-filter of $E$ such that $X \in \sigma$ and $\rho \subseteq \sigma$.
Therefore $\rho = \sigma$ by the maximality of $\rho$ and so $X \in \rho$. This proves our claim.
Now to prove that $\rho$ is an $s$-filter, take two open subsets $Y$ and $Z$ of $E$ such that $Y \not\in \rho$ and $Z \not\in \rho$.
Then $Y \not\in \mu$ and $Z \not\in \mu$. It follows that there exist $A$, $B \in \rho$ such that $Y \cap A = \emptyset$ and $Z \cap B = \emptyset$.
This implies that $(A \cap B) \cap (Y \cup Z) = \emptyset$. Since $A \cap B \in \rho$, it follows from the definition of a $\tau$-filter that
$Y \cup Z \not\in \rho$.
\end{proof}

In the following result, we characterize maximal filters on a space equipped with the discrete topology. We include its proof for readers' convenience.
This result shows that the converse of the preceding proposition holds true for filters on a discrete space.

\begin{proposition} \label{discrete} Let $E$ be a discrete space and let $\theta$ be a filter $($$\tau$-filter$)$ on $E$. Then the following
conditions are equivalent:

{\em (i)} $\theta$ is maximal;

{\em (ii)} $\theta$ is an $s$-filter;

{\em (iii)} For each subset $A \subseteq E$, either $A \in \theta$ or $A^c \in \theta$;

{\em (iv)} Given any subset $A \subseteq E$, $A \cap X \neq \emptyset$ for all $X \in \theta$ implies $A \in \theta$.
\end{proposition}

\begin{proof} (i) $\Rightarrow$ (ii) This follows from Proposition \ref{maximal-filter}.

(ii) $\Rightarrow$ (iii) This follows from the fact that $A \cup A^c=E \in \theta$ for every $A \subseteq E$ and that $\theta$ is saturated.

(iii) $\Rightarrow$ (iv) Let $A \subseteq E$ such that $A \not\in \theta$. Then $A^c \in \theta$ by (iii). Note that $A \cap A^c=\emptyset$.

(iv) $\Rightarrow$ (i) Suppose that there exists a filter $\rho$ on $E$ such that $\theta \subseteq \rho$ and $\theta \neq \rho$.
Take $A \in \rho \setminus \theta$. It is clear that $A \cap X \neq \emptyset$ for every $X \in \theta$. Therefore $A \in \theta$ by (iv), a contradiction.
This proves that $\theta$ is a maximal filter.
\end{proof}

The next result is a sharpened version of the equivalence (i) $\Leftrightarrow$ (iii) in Proposition \ref{discrete}.

\begin{proposition} \label{maximal-filter-2} Let $\theta$ be a $\tau$-filter of a topological space $E$. The following conditions are equivalent:

{\em (i)} $\theta$ is maximal;

{\em (ii)} For every open subset $A$ of $E$, either $A \in \theta$ or $(\overline{A})^c \in \theta$.
\end{proposition}

\begin{proof} (i) $\Rightarrow$ (ii) Let $A$ be an open subset of $E$ such that $A \not\in \theta$. By the proof of Proposition \ref{maximal-filter},
there exists $W \in \theta$ such that $A \cap W = \emptyset$. This yields $W \subseteq E \setminus A = A^c$.
Therefore $W \subseteq \mathring{\widehat{A^c}} = (\overline{A})^c$ and hence $(\overline{A})^c \in \theta$.

(ii) $\Rightarrow$ (i) Suppose that there exists a $\tau$-filter $\rho$ such that $\theta \subseteq \rho$ and $\rho \neq \theta$.
Let $A \in \rho$ with $A \not\in \theta$. Then $(\overline{A})^c \in \theta$ by (ii). This implies that $(\overline{A})^c \in \rho$.
Therefore $\emptyset = A \cap (\overline{A})^c \in \rho$, a contradiction. This shows that $\theta$ is a maximal $\tau$-filter.
\end{proof}

\begin{remark} \rm (i) Using Zorn Lemma, we infer that every $\tau$-filter is contained in a maximal $\tau$-filter.
Hence every $\tau$-filter is contained in an $s$-filter by Proposition \ref{maximal-filter}.

(ii) The converse of Proposition \ref{maximal-filter} is not true, in general. To illustrate this, we consider the set $\omega(a)$ of open neighbourhoods
of some $a \in \mathbb{R}$. Clearly,  $\omega(a)$ is an $s$-filter (see Example \ref{s-filter}). On the other hand, since $A=]a, a+1[ \not\in \omega(a)$
and $(\overline{A})^c \not\in \omega(a)$, it follows from Proposition \ref{maximal-filter-2} that the $\tau$-filter $\omega(a)$ is not maximal.

(iii) It is clear that if a set $E$ is equipped with the discrete topology, then $\widetilde{E}$ (the set of all $s$-filters on $E$) coincides with the Stone space of ultrafilters on $E$. In particular, $\widetilde{\mathbb{N}} = \beta\mathbb{N}$ where $\beta\mathbb{N}$ denotes the set of all ultrafilters on $\mathbb{N}$.
\end{remark}

Given a topological space $(E; \tau)$ and $A \in \tau$, we denote by $\widetilde{A}$ the set of all $s$-filters $\theta$ with $A \in \theta$.
The next proposition which is presumably well known is easy to prove. This result provides three basic useful facts concerning $s$-filters.

\begin{proposition} Let $(E; \tau)$ be a topological space and let $A \in \tau$. Then the following hold:

{\em (i)} The set $\{\widetilde{A} \mid A \in \tau\}$ is a base for the canonical topology on $\widetilde{E}$.

{\em (ii)} The set $[E]$ of principal $s$-filters is dense in $\widetilde{E}$.

{\em (iii)} The application $\varphi: E \rightarrow \widetilde{E}$ defined by $\varphi(x)=\omega(x)$ for all $x \in E$
is continuous on $E$. Moreover, $\varphi$ is injective if and only if $E$ is a $T_0$ space. In this case, $E$ can be identified with $[E]$.
\end{proposition}

\begin{theorem} \label{f-tilde} Let $(E; \tau)$ and $(F; \delta)$ be two topological spaces and let $f:E \rightarrow F$ be a continuous application.
Consider the application $\widetilde{f}: \widetilde{E} \rightarrow \widetilde{F}$ defined by $\widetilde{f}(\theta)=\{B \in \delta \mid f^{-1}(B) \in \theta\}$.
Then $\widetilde{f}$ is a well defined continuous application. Moreover, we have $\widetilde{f}(\omega(x))=\omega(f(x))$ for every $x \in E$.
\end{theorem}

\begin{proof} To prove that $\widetilde{f}$ is continuous, it suffices to show that $\widetilde{f}^{-1}(\widetilde{B})=\widetilde{f^{-1}(B)}$
for every open subset $B$ of $F$. Take $B \in \delta$. Then $$\theta \in \widetilde{f}^{-1}(\widetilde{B}) \Leftrightarrow \widetilde{f}(\theta) \in \widetilde{B} \Leftrightarrow B \in \widetilde{f}(\theta) \Leftrightarrow f^{-1}(B) \in \theta \Leftrightarrow \theta \in \widetilde{f^{-1}(B)}.$$
The remaining assertions are immediate.
\end{proof}

Recall that a topological space $(E; \tau)$ is called $T_3$ if for every $x \in E$, the set of all closed neighbourhoods of $x$ is
a fundamental system of neighborhoods of $x$.
Let $\theta \in \widetilde{E}$ and let $a \in E$. Then $a$ is said to be a {\it limit point} of $\theta$ if $\theta$ converges to $a$
(i.e. if $\omega(a) \subseteq \theta$). The element $a$ is called  an {\it adherent point} of $\theta$ if $a \in \cap_{A \in \theta} \overline{A}$.
Note that every limit point is an adherent point. An $s$-filter $\theta$ is called {\it free} if $\cap_{A \in \theta} \overline{A} = \emptyset$.

Next, we investigate some properties of $s$-filters on $T_3$ spaces.

\begin{proposition} \label{T3-space} Let $(E; \tau)$ be a $T_3$ space. Then every $s$-filter on $E$ is either free or convergent.
\end{proposition}

\begin{proof} Take an $s$-filter $\theta$ on $E$ and assume that $\theta$ is not free. Let $a \in \cap_{A \in \theta} \overline{A}$ and $V \in \omega(a)$.
Then there exists $W \in \omega(a)$ such that $W \subseteq \overline{W} \subseteq V$. Therefore $V \cup (\overline{W})^c=E \in \theta$.
Since $a$ is an adherent point of $\theta$ and $W \cap (\overline{W})^c = \emptyset$, it follows that $(\overline{W})^c \not\in \theta$.
This implies that $V \in \theta$ as $\theta$ is an $s$-filter. Consequently, $\omega(a) \subseteq \theta$, i.e. $\theta$ converges to $a$.
\end{proof}

\begin{corollary} If $(E; \tau)$ is a compact space, then every $s$-filter on $E$ is convergent.
\end{corollary}

\begin{proof} This follows from the fact that over a compact space $(E; \tau)$, every $\tau$-filter has an adherent point.
\end{proof}

\begin{proposition} \label{quasi-compact} Let $(E; \tau)$ be a topological space such that every $s$-filter on $E$ is convergent. Then $E$ is quasi-compact.
\end{proposition}

\begin{proof} Let $\delta$ be an ultrafilter on $E$. It is clear that the set $\delta \cap \tau$ of all open subsets of $E$ belonging to $\delta$ is an $s$-filter on $E$. By hypothesis, $\delta \cap \tau$ is convergent and hence $\delta$ is convergent.
\end{proof}

\begin{proposition} \label{quasi-compact-2} Let $(E; \tau)$ be a topological space. Then $\widetilde{E}$ is always quasi-compact.
\end{proposition}

\begin{proof} For each open subset $\Omega$ of $\widetilde{E}$, let $\Omega_E=\{x \in E \mid \omega(x) \in \Omega\}$. Now take an $s$-filter $\Gamma$
on $\widetilde{E}$. It is easily seen that $\theta=\{\Omega_E \mid \Omega \in \Gamma\} \in \widetilde{E}$. Moreover, $\Gamma$ converges to $\theta$.
Now the result follows from Proposition \ref{quasi-compact}.
\end{proof}

\begin{proposition} Let $E$ be a $T_1$ space $($i.e. every singleton in $E$ is closed$)$ and let $\theta$ be an $s$-filter on $E$.
Then $\cap_{A \in \theta} A \neq \emptyset$ if and only if $\theta$ is principal.
\end{proposition}

\begin{proof} The sufficiency is clear. Conversely, suppose that $\cap_{A \in \theta} A \neq \emptyset$ and let $a \in \cap_{A \in \theta} A$.
Then $\theta \subseteq \omega(a)$. Now take $V \in \omega(a)$. Therefore $\{a\}^c \cup V = E \in \theta$. Since $\{a\}^c \not\in \theta$ and $\theta$ is an
$s$-filter, we have $V \in \theta$. It follows that $\omega(a) \subseteq \theta$ and so $\theta = \omega(a)$ is principal.
\end{proof}

Let $(E; \tau)$ be a locally compact space which is not compact. In particular, $E$ is a regular space (and it is also a separated $T_3$ space).
The set $\{K^c \mid K$ is a compact subset of $E\}$ will be denoted by $\omega(\infty)$. Note that $\omega(\infty)$ is a free $\tau$-filter on $E$.

\begin{theorem} \label{converge-or-infinity} Let $E$ be a locally compact space which is not compact.
Then every $s$-filter on $E$ is either convergent or it is a free $s$-filter containing $\omega(\infty)$.
In the latter case we say that $\theta$ converges to infinity.
\end{theorem}

\begin{proof} Since $E$ is locally compact, $E$ is a $T_3$ space. According to Proposition \ref{T3-space}, every $s$-filter on $E$ is either free or convergent. Let $\theta$ be a free $s$-filter. To prove that $\omega(\infty) \subseteq \theta$, let $K$ be a compact subset of $E$. Since $\theta$ is free,
$\cap_{A \in \theta} \overline{A} = \emptyset$ and hence $\cap_{A \in \theta} (\overline{A} \cap K) = \emptyset$. Note that $\overline{A} \cap K$
is closed for each $A \in \theta$ because $E$ is separated. Since $K$ is compact, there exists a finite subset $\mathcal{F} \subseteq \theta$ such that
$\cap_{A \in \mathcal{F}} (\overline{A} \cap K) = (\cap_{A \in \mathcal{F}} \overline{A}) \cap K = \emptyset$.
This implies that $\cap_{A \in \mathcal{F}} A \subseteq \cap_{A \in \mathcal{F}} \overline{A} \subseteq K^c$.
Note that $\cap_{A \in \mathcal{F}} A \in \theta$ since $\mathcal{F}$ is finite. Then $K^c \in \theta$ and hence $\omega(\infty) \subseteq \theta$.
\end{proof}

\section{Hyperreals versus topological filters}

We equip the set $\mathbb{N}$ of non-negative integers with the discrete topology.
Let $x: \mathbb{N} \rightarrow \mathbb{R}$ be a real sequence defined by $x(n)=x_n$ for all $n \in \mathbb{N}$.
For every $U \in \widetilde{\mathbb{N}}$, we set $U(x)=\widetilde{x}(U)$ (see Theorem \ref{f-tilde}). That is,
\begin{center}
$U(x)=\{A \subseteq \mathbb{R} \mid A$ is an open subset of $\mathbb{R}$ with $\{n \in \mathbb{N} \mid x_n \in A\} \in U\} \in \widetilde{\mathbb{R}}.$
\end{center}

We next present an $s$-filter that is neither principal nor maximal.

\begin{example} \rm Consider the real sequence $x$ defined by $x_n=1/n$ for every $n \geq 1$. Let $U$ be a free ultrafilter on $\mathbb{N}$.
We have $\omega(0) \subseteq U(x)$ and $\omega(0) \neq U(x)$. Note that $U(x)$ is not maximal because $W=\cup_{n \geq 1} ]1/(n+1); 1/n[ \not\in U(x)$
and $(\overline{W})^c \not\in U(x)$.
\end{example}

Let $U$ be a free $s$-filter in $\widetilde{\mathbb{N}}$. Two real sequences $x$ and $y$ are said to be {\it $U$-equivalent} if $\{n\in \mathbb{N} \mid x_n=y_n\} \in U$. This relation is an equivalence relation on the set of real sequences. Let $\overline{x}$ denote the equivalence class of $x$. Note that the set
$\{\overline{x} \mid x \in \mathbb{R}^{\mathbb{N}}\}$ of equivalence classes coincides with the set of hyperreals (modulo $U$) (see \cite[p. 25]{Goldblatt}).
It is easy to check that $\overline{x}=\overline{y}$ implies $U(x)=U(y)$. To study the converse of this implication,
we need some notations and results. For any $a \in \mathbb{R}$, we consider the following sets:
\begin{center}
$\omega^{+}(a)=\{W$ open subset of $\mathbb{R}$ $\mid$ $]a; a+\varepsilon[ \subseteq W$ for some $\varepsilon > 0\}$,
\end{center}

\begin{center}
$\omega^{-}(a)=\{W$ open subset of $\mathbb{R}$ $\mid$ $]a-\varepsilon; a[ \subseteq W$ for some $\varepsilon > 0\}$,
\end{center}

\begin{center}
$\omega^{+}(\infty)=\{W$ open subset of $\mathbb{R}$ $\mid$ $]M; +\infty[ \subseteq W$ for some $M > 0\}$, and
\end{center}

\begin{center}
$\omega^{-}(\infty)=\{W$ open subset of $\mathbb{R}$ $\mid$ $]-\infty; -M[ \subseteq W$ for some $M > 0\}$.
\end{center}

\begin{proposition} \label{plus-minus} For each $\theta \in \widetilde{\mathbb{R}} \setminus \mathbb{R}$, there exists a unique
$a \in \mathbb{R} \cup \{\infty\}$ such that either $\omega^{+}(a) \subseteq \theta$ or $\omega^{-}(a) \subseteq \theta$.
\end{proposition}

\begin{proof} By Theorem \ref{converge-or-infinity}, there exists a unique $a \in \mathbb{R} \cup \{\infty\}$ such that $\omega(a) \subseteq \theta$.
Since $\theta$ is not principal, there exists $A \in \theta \setminus \omega(a)$.

\noindent {\bf Case 1:} $a$ is a real number. Since $a \not \in A$, we have
$$A \cap ]a-1; a+1[ = (A \cap ]a-1; a[) \cup (A \cap ]a; a+1[) \in \theta.$$
But $\theta$ is saturated, so $A \cap ]a-1; a[ \in \theta$ or  $A \cap ]a; a+1[ \in \theta$. Without loss of generality we can assume that
$A \cap ]a-1; a[ \in \theta$. This implies that $]a-1; a[ \in \theta$. To show that $\omega^{-}(a) \subseteq \theta$, take $W \in \omega^{-}(a)$.
Then $]a-\varepsilon; a[ \subseteq W$ for some $\varepsilon > 0$. If $\varepsilon \geq 1$, then $]a-1; a[ \subseteq ]a-\varepsilon; a[$
and hence $]a-\varepsilon; a[ \in \theta$. If $\varepsilon < 1$, then $]a-\varepsilon; a[ = ]a-1; a[ \cap ]a-\varepsilon; a+\varepsilon[ \in \theta$.
It follows that $W \in \theta$ and so $\omega^{-}(a) \subseteq \theta$.

\noindent {\bf Case 2:} $a=\infty$. The proof of this case runs as before by using the fact that
$$A \cap [-1; 1]^c = (A \cap ]-\infty; -1[) \cup (A \cap ]1; +\infty[) \in \theta.$$
\end{proof}

Choquet introduced the notion of absolute ulrafilters in 1968 \cite{Choquet-2}. A nontrivial ultrafilter $U$ on $\mathbb{N}$ is said to be {\it absolute}
if for every application $f: \mathbb{N} \rightarrow \mathbb{N}$ such that $\widetilde{f}(U)$ is not trivial, we have $\widetilde{f}(U)=U$.
Choquet showed that the Continuum Hypothesis guarantees the existence of such an ultrafilter and established the following result
(see \cite{Choquet-2}).

\begin{theorem}\label{a-ultrafilter-sequence} Let $U$ be a nontrivial ultrafilter. Then the following conditions are equivalent:
\begin{enumerate}
\item[{\em (i)}] $U$ is an absolute ulrafilter;
\item[{\em (ii)}] For every partition $(Q_n)_{n \geq 0}$ of $\mathbb{N}$ with $Q_n \not\in U$ for all $n \geq 0$, there exists $A \in U$ such that
$|Q_n \cap A| \leq 1$ for all $n \geq 0$.
\end{enumerate}
\end{theorem}

The next result which is taken from \cite{Choquet-1}.

\begin{theorem} \label{A-cap-f(A)} Let $U$ be an ultrafilter on a set $E$ and let $f:E \rightarrow E$ be an application.
Then there exists $A \in U$ such that exactly one of the following two conditions is satisfied:

{\em (i)} $f(x)=x$ for every $x \in A$, or

{\em (ii)} $A \cap f(A) = \emptyset$.
\end{theorem}

As a consequence of the preceding theorem, we obtain the following useful corollary.

\begin{corollary} \label{corol-Th2-Choquet} Let $U$ be an ultrafilter on a set $E$ and let $P \in U$. Assume that there exists an application
$f: P \rightarrow E$ such that
$$\forall S \in U, S \subseteq P \Rightarrow f(S) \in U.$$
Then there exists $A \in U$ such that $A \subseteq P$ and $f(x)=x$ for every $x \in A$.
\end{corollary}

\begin{proof} Let $g: E \rightarrow E$ be the extension of $f$ defined by $g(x)=x$ for every $x \in E \setminus P$.
From Theorem \ref{A-cap-f(A)}, it follows that there exists $Q \in U$ such that $Q \cap g(Q) = \emptyset$ or $g(x)=x$ for every $x \in Q$.
Hence the proof is done by setting $A=Q \cap P$ and using the fact that $A \cap f(A) \neq \emptyset$.
\end{proof}

In the remainder of this section we assume $U$ to be an absolute ultrafilter on $\mathbb{N}$.
We are ready to prove the main results of this paper. The first one which is maybe unexpected connects hyperreals to topological filters.

\begin{theorem} \label{converse} Let $U$ be an absolute ultrafilter on $\mathbb{N}$ and let $x$ and $y$ be two real sequences.
Then $U(x)=U(y)$ if and only if $\overline{x}=\overline{y}$.
\end{theorem}

\begin{proof} The sufficiency is clear. Conversely, suppose that $U(x)=U(y)$.

\noindent {\bf Case 1:} Assume that $U(x) \in \mathbb{R}$. Then $U(x)=\omega(a)$ for some real number $a$. We claim that $\overline{x}=\overline{a}$.
Suppose the contrary and let $P=\{n \geq 0 \mid x_n \neq a\}$. It is clear that $P \in U$. Let $W$ be an open subset of $\mathbb{R}$ such that $x_n \in W$
for every $n \in P$ but $a \not\in W$. Then $W \in U(x) \setminus \omega(a)$, a contradiction. It follows that $\overline{x}=\overline{y}=\overline{a}$.

\noindent {\bf Case 2:} Now assume that $U(x) \not\in \mathbb{R}$. By Proposition \ref{plus-minus}, there exists a unique $a \in \mathbb{R} \cup \{\infty\}$ such that either $\omega^{+}(a) \subseteq U(x)$ or $\omega^{-}(a) \subseteq U(x)$.
For simplicity of notation, if $a \in \mathbb{R}$ set
\begin{center}
$I_k=]a+\frac{1}{k+1}; a+\frac{1}{k}[$ and $Q_k=\{n \geq 0 \mid x_n \in ]a+\frac{1}{k+1}; a+\frac{1}{k}[\}$
\end{center}
for any positive integer $k$,

\begin{center}
$Q=\{n \geq 0 \mid x_n \in ]a; a+1[\}$ and $Q_0 = Q^c$.
\end{center}
If $a=\infty$, set
\begin{center}
$I_k=]k; k+1[$ and $Q_k=\{n \geq 0 \mid x_n \in ]k; k+1]\}$
\end{center}
for any positive integer $k$,
\begin{center}
$Q=\{n \geq 0 \mid x_n \in ]1; +\infty[\}$ and $Q_0 = Q^c$.
\end{center}
Note that $Q \in U$ and $Q_n \not\in U$ for every $n \geq 0$. Moreover, $(Q_n)_{n\geq 0}$ is a partition  of $\mathbb{N}$ satisfying the conditions
of Theorem \ref{a-ultrafilter-sequence}. The sequence $y$ can be handled in much the same way and then there exists $A \in U$ such that
the following two conditions hold:

\begin{enumerate}
\item[(1)] $\forall k \in A, \exists \alpha(k) \geq 1, \exists \beta(k) \geq 1, x_k \in I_{\alpha(k)}$ and $y_k \in I_{\beta(k)}$.

\item[(2)] For every $n \geq 1$, $I_n$ contains at most one element of $x$ and at most one element of $y$.
\end{enumerate}

For ease of notation set $B(z, r)=]z-r; z+r[$ for any $z \in \mathbb{R}$ and any positive real $r$.
Let $P=\{i \in A \mid \exists j = \varphi(i) \in A, x_i=y_j\}$, $E=\{x_i \mid i \in A\}$, $F=\{y_j \mid j \in A\}$ and $G=E \cup F$.
Taking into account the distribution of the points of $G$ on the intervals $I_n$ $(n \geq 1)$, we see that the mapping $\varphi:P \rightarrow \mathbb{N}$ is well defined. Let $R=\varphi(P)$ and consider the following neighbourhoods of the points of $G$.

If $i \in P$ and $j=\varphi(i) \in R$, $\exists \varepsilon_i=\mu_j > 0$, $B(x_i, \varepsilon_i) \cap G=\{x_i\}$.

If $i \not\in P$, $\exists \varepsilon_i > 0$, $B(x_i, \varepsilon_i) \cap G=\{x_i\}$.

\noindent Similarly,

if $j \not\in R$, $\exists \mu_j > 0$, $B(y_j, \mu_j) \cap G=\{y_j\}$.

For $M \subseteq A$, set $V_M=\cup_{i \in M} B(x_i, \varepsilon_i)$ and $W_M=\cup_{j \in M}B(y_j, \mu_j)$.
Note that $$M \in U \Leftrightarrow V_M \in U(x) \Leftrightarrow W_M \in U(y).$$
Since $U(x)=U(y)$, we must have $P \in U$ and for every $S \subseteq P$, $S \in U$ implies $\varphi(S) \in U$. In fact,

\begin{center}
$P \not\in U \Rightarrow T=P^c \cap A \in U \Rightarrow V_T \in U(x) \setminus U(y)$ and $W_T \in U(y) \setminus U(x)$, and

$S \in U \Rightarrow V_S=W_{\varphi(S)} \in U(x)=U(y) \Rightarrow \varphi(S) \in U$.
\end{center}
Now using Corollary \ref{corol-Th2-Choquet}, it follows that there exists $H \in U$ such that $\varphi(i)=i$ for every $i \in H$.
This clearly forces $\overline{x}=\overline{y}$.
\end{proof}

Recall that the set of hyperreals (modulo $U$) is denoted by $^\ast\mathbb{R}$. Then the application $h:$ $^\ast\mathbb{R} \rightarrow \widetilde{\mathbb{R}}$ defined by $h(\overline{x})=U(x)$ is injective. Hence the set $^\ast\mathbb{R}$ can be identified to a subset of $\widetilde{\mathbb{R}}$ and then it can be endowed with the induced topology.

\begin{theorem} \label{separated} The set $^\ast\mathbb{R} \setminus \mathbb{R}$ endowed with the induced topology by the space
$\widetilde{\mathbb{R}}$ is a separated topological space.
\end{theorem}

\begin{proof} Let $\overline{x}, \overline{y} \in$ $^\ast\mathbb{R} \setminus \mathbb{R}$ with $\overline{x} \neq \overline{y}$.
Then by Proposition \ref{plus-minus}, there exist $a, b \in \mathbb{R} \cup \{\infty\}$ and $\alpha, \beta \in \{+, -\}$ such that $\omega^{\alpha}(a) \subseteq U(x)$ and $\omega^{\beta}(b) \subseteq U(y)$. The proof falls naturally into three parts:
(i) $a\neq b$, (ii) $a=b$ and $\alpha \neq \beta$ and (iii) $a=b$ and $\alpha = \beta$.

For (i) and (ii), it is clear that there exist two disjoint open sets $V$ and $W$ such that $V \in U(x)$ and $W \in U(y)$.
Now assume that $a=b$ and $\alpha = \beta = +$. We will use the same notation of the proof of Theorem \ref{converse}.
Note that $U$ is an absolute ultrafilter.
Applying Theorem \ref{a-ultrafilter-sequence}, we conclude that there exists $A \in U$ satisfying the following two conditions:
\begin{enumerate}
\item[(1)] every interval $I_n$ contains at most an element $x_i$ and at most an element $y_j$ with $i, j \in A$, and
\item[(2)] for each $k, l \in A$, $\exists m, n \geq 0$, $x_k \in I_n$ and $y_l \in I_m$.
\end{enumerate}
\noindent Let $P=\{i \in A, \exists j=\varphi(i) \in A, x_i=y_j\}$ and set $R=\varphi(P)=\{j\in A, \exists i \in A, x_i=y_j\}$. It is easily seen that
$\varphi: P \rightarrow R$ is a bijection.

\noindent {\bf Case 1:} Assume that $P, Q \in U$. Since $\overline{x} \neq \overline{y}$, it follows from Corollary \ref{corol-Th2-Choquet}
that there exists $T \in U$ with $T \subseteq P$ and $\varphi(T) \not\in U$. This implies that $S=R \cap (\varphi(T))^c \in U$.
In this case $V_T$ and $W_S$ are two disjoint open sets such that $V_T \in U(x)$ and $W_S \in U(y)$.

\noindent {\bf Case 2:} Assume now that $P \not\in U$. Then $T=A \cap P^c \in U$. Therefore $V_T$ and $W_A$ are two disjoint open sets such that
$V_T \in U(x)$ and $W_A \in U(y)$. This completes the proof.
\end{proof}

\noindent {\bf Acknowledgement.}

The author would like to express his gratitude to Professor Labib Haddad of Clermont-Ferrand University for his kind help and many stimulating discussions
during the preparation of this paper. His feedback had led to substantial improvements on the quality of this work.
The author would also like to thank him for providing several interesting papers related to the study.
The author also wish to thank Professor Rachid Tribak for translating an earlier version of this paper from French into English, for preparing
this manuscript in LaTeX and for improving the presentation of the paper.

\end{document}